\documentclass{ws-ijnt}

\usepackage{graphicx}
\usepackage{hyperref}
\usepackage{ellipsis}
\usepackage{enumerate}
\usepackage{comment}
\usepackage[font=small,labelfont=bf]{caption}
\usepackage{subcaption}
\usepackage{url}
\usepackage{cite}
\usepackage{color}
\usepackage{breakurl}
\usepackage{comment}
\newcommand{\bburl}[1]{\textcolor{blue}{\url{#1}}}

\def\Z{{\mathbb{Z}}}

\newcommand{\eps}{\epsilon}

\newcommand{\fracloga}{\frac{\log |A+A|}{\log|A-A|} }
\newcommand{\tfracloga}{\tfrac{\log |A+A|}{\log|A-A|} }
\newcommand{\kbar}{\overline{k}}

\numberwithin{equation}{section}
\newcommand\be{\begin{equation}}
\newcommand\ee{\end{equation}}
\newcommand\bea{\begin{eqnarray}}
\newcommand\eea{\end{eqnarray}}

\begin{document}

\markboth{M.~Asada, S.~Manski, S.J.~Miller, H.~Suh}
{Fringe pairs in generalized MSTD sets}

%
\catchline{}{}{}{}{}
%

\title{FRINGE PAIRS IN GENERALIZED MSTD SETS}

\author{MEGUMI ASADA}

\address{{Department of Mathematics and Statistics, Williams College, Williamstown, MA 01267}\\
\email{\href{mailto:maa2@williams.edu}{maa2@williams.edu}} }

\author{SARAH MANSKI}

\address{{Department of Mathematics \& Computer Science, Kalamazoo College, Kalamazoo, MI 49006}\\
	\email{\href{mailto:Sarah.Manski12@kzoo.edu}{Sarah.Manski12@kzoo.edu}} }

\author{STEVEN J. MILLER}

\address{{Department of Mathematics and Statistics, Williams College, Williamstown, MA 01267}\\
	\email{{\href{mailto:sjm1@williams.edu}{sjm1@williams.edu}},  {\href{Steven.Miller.MC.96@aya.yale.edu}{Steven.Miller.MC.96@aya.yale.edu}} }}
	
\author{HONG SUH}

\address{{Department of Mathematics, University of California, Berkeley, CA 94720}\\
	\email{\href{mailto:hong.suh@berkeley.edu}{hong.suh@berkeley.edu}} }

\maketitle

\begin{history}
\received{???}
\accepted{20 October 2016}
\end{history}

\begin{abstract}
A \emph{More Sums Than Differences} (MSTD) set is a set $A$ for which $|A+A|>|A-A|$. Martin and O'Bryant proved that the proportion of MSTD sets in $\{0,1,\ldots,n\}$ is bounded below by a positive number as $n$ goes to infinity. Iyer, Lazarev, Miller and Zhang introduced the notion of a \emph{generalized MSTD set}, a set $A$ for which $|sA-dA|>|\sigma A-\delta A|$ for a prescribed $s+d=\sigma+\delta$. We offer efficient constructions of $k$-generational MSTD sets, sets $A$ where $A, A+A, \dots, kA$ are all MSTD. We also offer an alternative proof that the proportion of sets $A$ for which $|sA-dA|-|\sigma A-\delta A|=x$ is positive, for any $x \in \Z$. We prove that for any $\eps>0$, $\Pr(1-\eps<\log |sA-dA|/\log|\sigma A-\delta A|<1+\eps)$ goes to $1$ as the size of $A$ goes to infinity and we give a set $A$ which has the current highest value of $\log |A+A|/\log |A-A|$. We also study decompositions of intervals $\{0,1,\ldots,n\}$ into MSTD sets and prove that a positive proportion of decompositions into two sets have the property that both sets are MSTD.
\end{abstract}


\section{Introduction}\label{sec:intro}
A \emph{More Sums Than Differences} (MSTD) set is a set $A$ where the sumset of $A$, denoted $A+A \ = \ \lbrace a_1+a_2 : a_1, a_2 \in A \rbrace$, has a greater cardinality than the difference set, $A-A \ = \ \lbrace a_1-a_2 : a_1, a_2 \in A \rbrace$. Since addition is commutative while subtraction is not, it is reasonable to expect the difference set to be larger than the sumset in most cases; such sets are said to be difference-dominated, while sets where $|A+A| = |A-A|$ are called sum-difference balanced. Interestingly, while most sets of $\{0, 1, \dots, n\}$ are difference dominated as $n\to\infty$, Martin and O'Bryant \cite{MO} proved that a positive percentage \emph{are} sum-dominated. They predict an expected limiting density of MSTD sets of about $4.5 \times 10^{-4}$; the best lower bound at the time of writing is due to Zhao \cite{Zh1}: $4.28 \times 10^{-4}$.

Conway found the first example of an MSTD set \[\lbrace 0, 2, 3, 4, 7, 11, 12, 14 \rbrace\] in the 1960's. Hegarty later proved that there were no MSTD sets of smaller cardinality \cite{He}. While some work has been generalized to MSTD sets in abelian groups \cite{MV,Na4,Zh3}, we concern ourselves with MSTD sets $A$ with $A \subset \mathbb{Z}$. Some progress has been made in constructing infinite families of such sets \cite{Na1,Na2,Na3,Na4,Zh1,He,MOS,MS,Ru1,Ru2}, as well as MSTD sets with additional properties. One example is due to Iyer, Lazarev, Miller and Zhang \cite{ILMZ}, who developed constructions for \emph{generalized MSTD sets}, sets $A$ that satisfy $|sA-dA|>|\sigma A - \delta A|$ for a given $s+d=\sigma + \delta$. This provided them with the framework to construct $k$-generational sets, a rarer class of MSTD sets $A$ for which $A, A+A, \ldots, kA$ are all sum-dominant.

A unifying strategy used throughout this paper is the manipulation of \emph{fringe pairs}, a concept introduced by Zhao \cite{Zh2} (see Definition \ref{defn:fringepair}). Our first set of results concern fringe constructions for generalized MSTD sets that are much more efficient than the ones in \cite{ILMZ}. Before stating these we first set some notation:


\begin{itemize}
	\item $|A|$ is the cardinality of $A$,
	\item $[a,b] \ = \ \{a,a+1,\ldots,b\}$,
	\item $mA \ = \ \{\sum_{i=1}^m a_i : a_i \in A\}$ for $m \geq 1$ (note this is \emph{not} $m$ times each element of $A$),
	\item $m \cdot A \ =  \ \{ma : a \in A\}$,
	\item $-A \ = \ \{-a : a \in A\}$, $-mA=-(mA)$,
	\item $A^c \ = \ [\min A, \max A] \setminus A$.
\end{itemize}

Our first result significantly improves upon the constructions of $k$-generational sets given by Iyer et al. \cite{ILMZ}, which used base expansion and therefore yielded sets $A$ which grow astronomically with $k$.

\begin{theorem}\label{thm:efficientAOqsquared}
	For any $q>2$, there exists a set $A$ with $|A|=O(q^2)$ such that for all $s+d=\sigma +\delta $ with $s>\sigma $,  \be |sA-dA|\ >\ |\sigma A-\delta A|.\ee Further, given $k>0$ there exists a $k$-generational set $A$ with $|A|=O(k)$.
\end{theorem}

We prove Theorem \ref{thm:efficientAOqsquared} in \S\ref{pf:efficientkgen}. Our constructions generate explicit $k$-generational sets that are much smaller than those previously constructed, making them far easier to manipulate and verify. We then apply the fringe pairs framework to give a new proof of the following result in \S\ref{pf:proportionsdsigmadelta}.

\begin{theorem}\label{thm:proportionsdsigmadelta}
	For any $x \in \Z$ and $0\leq d <\delta \leq \sigma <s$ with $s+d=\sigma+\delta$, the proportion of $A\subseteq [0,n]$ such that $|sA-dA|-|\sigma A- \delta A|=x$ is bounded below by a positive number as $n$ goes to infinity.
\end{theorem}

We also study the values of $\log |sA-dA| / \log |\sigma A-\delta A|$. These are natural objects to study as they normalize the excess of one generalized sum or difference set relative to the number of elements of that set. In other words, if we want to construct a set with many more sums than differences we don't want to do so through base expansion and taking exponentially large sets; we want to construct small sets with the desired excess. The quantity $\log |A+A| / \log |A-A|$ was studied by Ruzsa in \cite{Ru2}, where he proved that there are a ``multitude'' of sets with $\log |A+A|/\log |A-A|>1+c$ with $c>0$. Precisely, he proved that there exists $c>0$ such that for every sufficiently large $n$ there exists a set such that \[|A|\ = \ n,\ \ \ |A-A|\ \leq\ n^{2-c},\ \ \ |A+A| \ \geq\ n^2/2-n^{2-c},\] or in other words, $\log|A+A|/\log|A-A| > 1+c_1$ where $c_1>0$. We show that for any $\eps>0$ there are not enough sets $A\subseteq [0,n]$ with $\log|sA-dA|/\log|\sigma A-\delta A|>1+\eps$ to constitute a positive proportion. In fact, we prove a stronger statement about $|sA-dA|/|\sigma A-\delta A|$ in \S\ref{pf:strongerratio}.

\begin{theorem}\label{thm:strongerratio} Choosing subsets $A \subset [0,n]$ uniformly, for every $\epsilon>0$ we  have
	\begin{equation} \label{eq:log1}
	\lim_{n\to\infty} \Pr \left(1-\eps< \frac{\log|sA-dA|}{\log|\sigma A-\delta A|}<1+\eps \right)=1.
	\end{equation}
	In fact,
	\begin{equation} \label{eq:log2}
	\lim_{n\to\infty} \Pr \left( \frac{s+d-1}{s+d}< \frac{|sA-dA|}{|\sigma A-\delta A|}< \frac{s+d}{s+d-1} \right) \ = \ 1.\end{equation}
\end{theorem}

We also construct a set $A$ for which \be \fracloga \  = \ \frac{\log 892}{\log 765}\ = \ 1.02313,\ee which is larger than the previous largest value of \be \fracloga\ = \ \frac{\log 91}{\log 83} \ = \ 1.0208\ee found in \cite{He, MO}.

Finally, we investigate decompositions of the interval $[0,n]$ into two disjoint MSTD sets in \S\ref{sec:bimstd}. It turns out that a positive proportion of decompositions of $[0,n]$ into two disjoint sets have the property that both components are MSTD. If both $A \subseteq [0,n]$ and $[0,n] \setminus A$ are MSTD, we say that $A$ is \emph{bi-MSTD}.

\begin{theorem}\label{thm:bimstd}
	Let $A$ be a uniform random subset of $[0,n]$ with $n \geq 19$. Then the probability that $A$ is bi-MSTD is bounded below by a positive number (and there are no bi-MSTD sets for $n < 19$).
\end{theorem}

The proof of Theorem \ref{thm:bimstd} is given in \S\ref{subsec:bimstd}.

\ \\

The paper is organized as follows. We define fringe pairs and prove some needed properties in \S\ref{sec:background}. We then prove our results on efficient $k$-generational constructions in \S\ref{sec:kgen}, on arbitrary differences in \S\ref{sec:arbitrarydiff}, on the ratio of the logarithms in \S\ref{sec:ratio} and bi-MSTD sets in \S\ref{sec:bimstd}. We conclude in \S\ref{sec:future} with topics for future investigations.

\section{Background}\label{sec:background}

\subsection{Definitions}

Martin and O'Bryant used the idea that the fringes of $A$ essentially determine whether or not $A$ is MSTD, as the middle does not contribute as much to $|A+A|-|A-A|$ since almost all middle elements of the sum and difference set are attained \cite{MO}. This is because the number of elements of a set $A$ chosen uniformly among subsets of $[0, n]$ is tightly concentrated around $(n+1)/2$, and there are so many ways of writing middle numbers as a sum or difference that with high probability all are realized. Zhao built on these ideas to prove that the proportion of MSTD sets in $[0,n]$ converges as $n$ goes to infinity. We extend many of Zhao's definitions and results in \cite{Zh2} to the setting of generalized MSTD sets. Of particular interest are Definition \ref{defn:fringepair} and Definition \ref{defn:richset}, which define \emph{fringe pairs} and \emph{rich sets} for generalized MSTD sets. The results in this section will prove valuable in the proofs of our main results in Sections \ref{pf:efficientkgen}, \ref{pf:proportionsdsigmadelta}, and \ref{subsec:bimstd}. In the following, we fix $s+d=\sigma +\delta $.

\begin{definition}\label{defn:fringepair} Let $s, d, \sigma, \delta$ be non-negative integers such that $0\leq d <\delta \leq \sigma <s$ and $s+d=\sigma+\delta$.  A generalized MSTD fringe pair of order $k$ for $(s,d), (\sigma ,\delta)$ is a pair $(L,R;k)$ such that $L,R \subseteq [0,k]$ with $0 \in L,R$ and \begin{eqnarray} & & |(sL+dR) \cap [0,k]| + |(sR+dL) \cap [0,k]| \nonumber\\ & & \ \ \ \ \ \ \ \ \ >\ |(\sigma L+\delta R) \cap [0,k]| + |(\sigma R+\delta L) \cap [0,k]|.\end{eqnarray}
\end{definition}

We impose an order on the set of generalized MSTD fringe pairs by having $(L,R;k)\leq (L',R';k')$ if $k\leq k'$ and \be L \ = \ L' \cap [0,k], \hspace{5mm} R=R' \cap [0,k], \hspace{5mm} [k+1,k'] \subseteq sL'+dR', sR'+dL'.\ee \label{def:minimalfpreq}

$(L,R;k) < (L',R';k')$ is defined as expected with $k < k'$ as well as the stipulations listed in \eqref{def:minimalfpreq}.

\begin{definition}\label{defn:minimalfringepair}
	A minimal generalized MSTD fringe pair is a pair $(L,R;k)$ such that $(L,R;k)\leq (L',R';k')$ for all generalized MSTD fringe pairs $(L',R';k')$ with $k' \ge k$.
\end{definition}

We later use the partial order of generalized MSTD fringe pairs to reduce the study of all generalized MSTD fringe pairs to simply the minimal ones.

\begin{definition}\label{defn:richset}
	A set $A \subseteq [0,n]$ is a $k$-\emph{rich set} with MSTD fringe pair $(L,R;k)$ if \begin{eqnarray} & & 2k \ < \ n, \ \ \ \ \ A \cap [0,k]\ = \ L, \ \ \  \ \ \ A \cap [n-k,n]\ = \ n-R, \nonumber\\ & & \ \ \ \ \ \ \ \ \ \ \ \ \ \ \ \ \  [k+1, 2n-k-1]  \ \subseteq\ A+A.\end{eqnarray} When we do not specify a fringe pair and say $A$ is $k$-rich, we often simply mean $[k+1,2n-k-1] \subseteq A+A$. Sometimes we simply say that $A$ is rich if $k$ is clear from the context.
\end{definition}

\subsection{Important Characteristics of Rich Sets}

It turns out that a rich set with generalized MSTD fringe pair $(L,R;k)$ is a generalized MSTD with respect to $(s,d)$, $(\sigma,\delta)$ (Lemma \ref{lemma:rich}). We first prove a simple lemma.

\begin{lemma}\label{lemma:A+A}
	A $k$-rich set $A\subseteq [0,n]$ has the property \be [-dn+k+1,sn-k-1]\ \subseteq\ sA-dA \ee for any $s+d>2$.
\end{lemma}

\begin{proof}
	Suppose without loss of generality that $s \geq d$. For $d \geq s$, simply switch the roles of $s$ and $d$. Since $s+d>2$ and $s \geq d$, we have $s \geq 2$. Thus $sA-dA=2A+(s-2)A-dA$. Then
	\begin{align}
	2A+(s-2)A-dA &\ \supseteq\ [k+1,2n-k-1]+(s-2)A-dA \nonumber\\
	&\ \supseteq\ [k+1,2n-k-1]+(s-2)\{0,n\}-d\{0,n\} \nonumber\\
	&\ =\ [-dn+k+1,sn-k-1]
	\end{align}
	since $[k+1,2n-k-1] \cup [n+k+1,3n-k-1] \cup \{2n\}= [k+1,3n-k-1]$ and so on.
\end{proof}

\begin{lemma}\label{lemma:rich} Let $0\leq d <\delta \leq \sigma <s$ and $s+d=\sigma+\delta$. A rich set $A \subseteq [0,n]$ with generalized MSTD fringe pair $(L,R;k)$ satisfies $|sA-dA|>|\sigma A-\delta A|$.
\end{lemma}
\begin{proof}
	We split $sA-dA$ into two portions, the fringe and the middle. We assert the following two statements, which are sufficient to prove the lemma:
	\begin{multline}\label{eq:fringe}
	|(sA-dA) \cap ([-dn,-dn+k] \cup [sn-k,sn]) | \\ >\ |(\sigma A-\delta A) \cap ([-\delta n,-\delta n+k] \cup [\sigma n-k,\sigma n])|
	\end{multline} and
	\begin{multline}\label{eq:middle}
	|(sA-dA) \cap [-dn+k+1,sn-k-1]| \\ \geq\ |(\sigma A-\delta A) \cap [-\delta n+k+1, \sigma n-k-1]|.
	\end{multline}
	Since $A$ is a rich set, $(sA-dA) \supseteq [-dn+k+1,sn-k-1]$. Thus \eqref{eq:middle} follows. To show \eqref{eq:fringe}, we notice that
	\begin{multline}
	|(sA-dA) \cap ([-dn,-dn+k] \cup [sn-k,sn]) | \\ = \ |(sA-dA) \cap [-dn,-dn+k]| + |(sA-dA) \cap [sn-k,sn]|.
	\end{multline}
	Since $A\subseteq [0,n]$ has generalized MSTD fringe pair $(L,R;k)$, we have $L=A \cap [0,k]$ and $R=n-(A \cap [n-k,n])$.  Using $\cong$ to denote equivalence under translation and dilation, we find
	\begin{align}
	(sL+dR) \cap [0,k] &\ =\ \left(s(A\cap[0,k]) + dn - d(A\cap[n-k,n]) \right) \cap [0,k] \nonumber\\
	&\ \cong\ \left(s(A \cap [0,k]) -d(A \cap [n-k,n])\right) \cap[-dn,-dn+k] \nonumber\\
	&\ =\ (sA-dA)\cap[-dn,-dn+k].
	\end{align}
	Similarly, we have
	\begin{align}
	(sR+dL) \cap [0,k] & \ =\ (sn-s(A \cap [n-k,n]) + d(A \cap [0,k]) ) \cap [0,k] \nonumber\\
	&\ \cong\ (s(A \cap [n-k,n]) - d(A \cap [0,k]) ) \cap [sn-k, sn] \nonumber\\
	&\ = \ (sA-dA) \cap [sn-k, sn].
	\end{align}
	Now \eqref{eq:fringe} follows from the definition of a generalized MSTD fringe pair.
\end{proof}


\subsection{Minimal Fringe Pairs}
Now we show that the order we imposed on fringe pairs earlier is completely determined by $k,k'$ for fringe pairs $(L,R;k)$ and $(L',R';k')$ corresponding to a rich set $A \subseteq [0,n]$.

\begin{lemma}
	Let $A \subseteq [0,n]$ be a rich set. Let $(L,R;k)$ and $(L',R';k')$ be two generalized MSTD fringe pairs corresponding to $A$. If $k=k'$, then $(L,R;k)=(L',R';k')$. If $k<k'$, then $(L,R;k)<(L',R';k')$.
\end{lemma}

\begin{proof}
	The first statement follows trivially. If $k<k'$, then we need to show $$L=L' \cap [0,k], \hspace{5mm} R=R' \cap [0,k], \hspace{5mm} [k+1,k'] \subseteq sL'+dR', sR'+dL'.$$ We have $A \cap [0,k] = L, A \cap [0,k'] = L'$, and the analogous statements for $R$ and $R'$. Since $k<k'$, $L=A \cap [0,k] = (A \cap [0,k']) \cap [0,k]=L' \cap [0,k]$. The same holds for $R$ and $R'$. Finally, we have $$[-dn+k'+1,sn-k'-1]\ \subseteq\ [-dn+k+1,sn-k-1]\ \subseteq\ sA-dA.$$ Thus $$[-dn+k+1,-dn+k']\ \subseteq\ sA-dA,\hspace{4mm} [sn-k-1,sn-k']\ \subseteq\ sA-dA.$$ An argument from the proof of Lemma \ref{lemma:rich} shows that $[-dn+k+1,-dn+k'] \subseteq sA-dA \iff [k+1,k'] \subseteq sL+dR$ and $[sn-k-1,sn-k']\subseteq sA-dA \iff [k+1,k']\subseteq sR+dL$. Therefore $(L,R;k)<(L',R';k')$.
\end{proof}

\begin{lemma}\label{lemma:minimal}
	Let $(L,R;k)$ be the minimal generalized MSTD fringe pair of a rich set $A\subseteq[0,n]$. Then $(L,R;k)$ is minimal in the partial ordering of all generalized MSTD fringe pairs. Also, for every $k<k'<n/2$, $(L',R';k')$ is also a generalized MSTD fringe pair of $A$, where $L'=A \cap [0,k'], R'=(n-A)\cap[0,k']$, and every generalized MSTD fringe pair of $A$ has this form.
\end{lemma}
\begin{proof}
	Suppose $(L,R;k)$ is not minimal. Then there is a generalized MSTD fringe pair $(L',R';k')<(L,R;k)$. So $k'<k$ and  $$L'=L \cap [0,k'], \hspace{5mm} R'=R \cap [0,k'], \hspace{5mm} [k+1,k'] \subseteq sL'+dR', sR'+dL'.$$ This implies that $[k+1,2n-k-1] \subseteq A+A$ and therefore $A$ is rich with generalized MSTD fringe pair $(L',R';k')$ as well, which contradicts the minimality of $(L,R;k)$ when attached to $A$.
	
	So $(L,R;k)$ is the minimal generalized MSTD fringe pair of $A$. Take $k<k'$. Then the only possible generalized MSTD fringe pair is $(L',R';k')$ where \[L' \ =\ A \cap [0,k'],\ \ \ R'\ = \ A \cap [0,k'].\] By the previous lemma, $(L',R';k')>(L,R;k)$. A simple computation confirms that $(L',R';k')$ is also a generalized MSTD fringe pair.
\end{proof}

The above results show that we may count all generalized MSTD sets by counting all minimal generalized MSTD fringe pairs.

\begin{remark} Iyer et al. \cite{ILMZ} proved that the proportion of generalized MSTD sets is bounded below by a positive number as $n$ goes to infinity. Zhao's methods should easily generalize to show that this proportion converges. Our definition of a rich set is identical to Zhao's. The only difference is the nature of the fringe pairs, but the fringe pairs are not involved in the proof of convergence.
	
	Zhao introduced \emph{$k$-affluent sets}, sets $A$ for which $[k+1,2n-k-1]\subseteq A+A$ and $[-n+k+1,n-k-1]\subseteq A-A$ (see Definition \ref{def:affluent}), in order to show that for any $m\in \Z$ the proportion of $A$ which satisfy $|A+A|-|A-A|=m$ converges as $n \to \infty$. The analogous result for generalized MSTD sets does not require affluent sets (we use affluent sets in this paper, albeit for a different purpose). Since we may assume that $s+d=\sigma +\delta \geq 3$, both $sA-dA$ and $\sigma A-\delta A$ contain a copy of $A+A$, up to a minus sign. Thus all the convergence results in \cite{Zh2} should be applicable to the generalized MSTD case. We do not pursue this path as it is too repetitive to merit detailing.
\end{remark}

\bigskip


\section{Efficient constructions of $k$-generational sets and their generalizations} \label{sec:kgen}


\subsection{Previous Constructions}
In \cite{ILMZ}, Iyer et al. gave a construction of \emph{$k$-generational sets}, which are sets $A$ such that $A, A+A, \ldots, kA$ are all MSTD. The primary tool was a technique called base expansion, summarized in the following proposition.

\begin{proposition}[Iyer et al. \cite{ILMZ}]
	Fix a positive integer $k$. Say $A_1,\ldots,A_k \subseteq \Z^+$. Choose some $m>k \cdot \max (\cup_{i=1}^k A_k)$. Let $C=A_1 + m \cdot A_2 + \cdots + m^{k-1} \cdot A_k$, where $\cdot$ denotes scalar multiplication. Then $$|sC-dC| \ = \     \prod_{j=1}^k|sA_j-dA_j|$$ for all $s+d\leq k$.
\end{proposition}

Using base expansion, we may choose $A_j$ such that $|jA_j+jA_j|>|jA_j-jA_j|$ (and $|sA_j-dA_j|=|\sigma A_j-\delta A_j|$ for $s+d=\sigma +\delta  \neq 2j$) and create the appropriate $C$ prescribed above, which is $k$-generational. However, a major drawback of base expansion is that the set $C$ grows large very quickly (we explore this issue in greater detail in \S\ref{sec:ratio}, where we investigate the ratio of the logarithms of the cardinalities). According to the construction in \cite{ILMZ}, the middle of $A_j$ has at least $2(2jr-4j+1)$ elements where $r=4j+2$. Thus $|A_j|=\Omega(j^2)$, which means there is a constant $c$ such that $c j^2 \le |A_j|$. By the proof of Lemma 4.3 in \cite{ILMZ}, we have $|C|=\prod_{j=1}^k |A_j|$. Therefore $|C|=\Omega(k!^2)$.

This is a huge growth rate that leads to sets which are computationally impractical to work with even when $k=2$. For example, an optimization of the construction with base expansion yields a 5-generational set of $2,685,375$ elements. Our construction reduces this size to $35$ (see Proposition \ref{prop:k-gen}). We are able to create much more reasonably sized sets by choosing appropriate fringe pairs and filling out the middle rather than using base expansion, which has no regard for the size of the set.


\subsection{Proof of Theorem \ref{thm:efficientAOqsquared}}\label{pf:efficientkgen}
We first prove the existence of a $k$-generational fringe pair in Proposition \ref{prop:kgenimprove}. In Proposition \ref{prop:k-gen}, we construct a rich set with the $k$-generational fringe pair from Proposition \ref{prop:kgenimprove}, giving us a $k$-generational set. We note that our rich set varies linearly with $k$ to complete the proof.

\begin{lemma} \label{0mq}
	Let $R=[0,m] \cup \{q\}$.  Then \[kR\ =\ [0,km] \cup [q, q+(k-1)m] \cup \cdots \cup[(k-1)q, (k-1)q +m]  \cup \{kq\}.\] Further, if $q>km$, then \[|kR|\ = \ \frac{(mk+2)(k+1)}{2}.\]
\end{lemma}

\begin{proof}
	Observe that \[k(A \cup B) \ = \ kA \cup ((k-1)A+B) \cup \cdots \cup (A+(k-1)B) \cup kB.\] Set $A=[0,m],B=\{q\}$ to obtain the result. Note then if $km<q$, $|kR| = \frac{(k)(k+1)m}{2} + k+1=\frac{(mk+2)(k+1)}{2}.$
\end{proof}

\begin{proposition} \label{prop:kgenimprove}
	Let $L=\{0\}, R=\{0,1,3\}$. Then the fringe pair $(L,R,6k)$ is $k$-generational in the sense that for all $1 \leq j \leq k$, $$|(2j)L \cap [0,6k]| + |(2j)R \cap [0,6k]|\ >\ 2|(jL+jR) \cap [0,6k]|.$$
\end{proposition}

\begin{proof}
	Fix $1 \leq j \leq k$. In Lemma \ref{0mq} set $m=1,q=3,k=2j$ to get $(2j)R=[0,6j]\setminus\{6j-1\}$. Also, $jL+jR=[0,3j]\setminus\{3j-1\}$. Therefore \[|(2j)L \cap [0,6k]| + |(2j)R \cap [0,6k]|\ =\ 1 + 6j, \hspace{4mm} 2|(jL+jR) \cap [0,6k]|=2\cdot 3j. \]
\end{proof}

\bigskip
By Lemma \ref{lemma:rich}, we may pick any rich set $A$ with fringe pair $(\{0\},\{0,1,3\};6k)$ and $A$ is $k$-generational. In Proposition \ref{prop:k-gen}, we explicitly construct such a set.

\bigskip

\begin{proposition}\label{prop:k-gen}
	Let $A=\{0\} \cup [6k+1,12k+1] \cup (18k+2-\{0,1,3\})$. Then $A$ is a rich set with $k$-generational fringe pair $(L,R,6k)$ and thus $k$-generational.
\end{proposition}
\begin{proof}
	It suffices to show that $[6k+1, 30k+3]\subseteq A+A$. Note that $|[6k+1,12k+1]|=6k+1$. Compute
	\begin{align*}
	A+A & \ \supseteq\ [6k+1,12k+1] \cup 2[6k+1,12k+1] \cup ([6k+1,12k+1]+\{18k+2\})\\
	& \ = \ [6k+1,30k+3]. 
	\end{align*}
\end{proof}

\bigskip
Since $A$ has $6k+5$ elements, the growth of the size of these $k$-generational sets $A$ is $|A|=O(k)$. This is a significant improvement over the previous best construction of $k$-generational sets, which have size $|A|=\Omega(k!^2)$.


\begin{remark}\label{rem:kgenex}
	This construction provides us with a nice 2-generational MSTD set with 17 elements,
	$$\lbrace 0, 13, 14, 15, 16, 17, 18, 19, 20, 21, 22, 23, 24, 25, 35, 37, 38 \rbrace,$$
	as well as a 3-generational set of size 23,
	$$\lbrace 0, 19, 20, 21, 22, 23, 24, 25, 26, 27, 28, 29, 30, 31, 32, 33, 34, 35, 36, 37, 53, 55, 56 \rbrace.$$
\end{remark}

\begin{remark}\label{rem:superkgen}
	One advantage of the base expansion method is that for any given sequence $\{(s_i,d_i,\sigma_i,\delta_i)\}_{i=2}^q$ with $s_i+d_i=\sigma_i+\delta_i=i$ and $s_i \neq \sigma_i$, $s_i \neq \delta_i$, the base expansion method can construct a set $A$ which satisfies $|s_iA-d_iA|>|\sigma_iA-\delta_iA|.$ Our construction is specific to $k$-generational sets. Though we do not reach the full level of generality achieved by Iyer et al., for any given sequence $\{(s_i,d_i,\sigma_i,\delta_i)\}_{i=2}^q$ with $s_i+d_i=\sigma_i+\delta_i=i$ \emph{and} $0\le d_i < \delta_i \le \sigma_i <s_i$, we describe an efficient construction of a set $A$ which satisfies $|s_iA-d_iA|>|\sigma_iA-\delta_iA|$ in \S\ref{sec:superkgen}. 
\end{remark}


\subsection{Super $k$-generational MSTD Sets} \label{sec:superkgen}
We prove the existence of a stronger form of $k$-generational sets, \emph{super $k$-generational sets},  in Proposition \ref{pf:superkgenexistence}.

\begin{definition} \label{def:superkgen}
	A \emph{super $k$-generational MSTD sets} is a set $A$ in which for all $s + d = \sigma + \delta \leq k$ with $0\le d < \delta \le \sigma <s$, $|sA - dA|>|\sigma A - \delta A|$.
\end{definition}

\begin{lemma}
	Consider the set $\{0,1,q\}$ with $q>2$.  We have $k\{0,1,q\} = [0,k] \cup [q, q+k-1] \cup [2q, 2q + k -2] \cup \cdots \cup \{kq\}$.
\end{lemma}
\begin{proof}
	This is a consequence of Lemma \ref{0mq}.
\end{proof}
Note that if $k<q$, then $|k\{0,1,q\}| = \frac{(k+1)(k+2)}{2}$.  Similarly if $k=q$ then $|k\{0,1,q\}| = \frac{(k+1)(k+2)}{2}-1.$

\begin{proposition} \label{pf:superkgenexistence}
	Let $L = \{0\}, R = \{0,1,q\}, k=q^2$ where $q>2$.  Then for any $s+d = \sigma  + \delta \le q$ with $0\le d < \delta \le \sigma <s$, $(L,R;k)$ is a generalized MSTD fringe pair. 
\end{proposition}
\begin{proof}
	\ \\
	Case 1: Let $d = 0$.  Observe that $sL = \{0\}, sR = [0,s] \cup [q, q+s-1] \cup \cdots \cup \{sq\}$.  Also $\sigma L + \delta R = [0,\delta ] \cup [q, q+\delta -1]\cup \cdots \cup \{\delta q\}$ and $\delta L + \sigma R = [0,\sigma ] \cup [q, q+\sigma -1]\cup \cdots \cup \{\sigma q\}.$  Therefore $$|(sL+dR) \cap [0,k]|+|(sR+dL) \cap [0,k]|\ \ge \ 1+\frac{(s+1)(s +2)}{2} -1$$ and $$|(\sigma L+\delta R) \cap [0,k]| + |(\sigma R+\delta L) \cap [0,k]|\ \le \ \frac{(\sigma +1)(\sigma +2)}{2} + \frac{(\delta +1)(\delta +2)}{2}.$$  
	Since $s>\sigma $ and $\sigma, \delta $ are nonzero, it is clear that $s^2 >\sigma ^2 + \delta ^2$ (remember $d=0$ in this case, so $s = \sigma + \delta$).  Thus \begin{equation*}\begin{split}|(sL+dR) \cap [0,k]|&+|(sR+dL) \cap [0,k]| \\ &>\ |(\sigma L+\delta R) \cap [0,k]| + |(\sigma R+\delta L) \cap [0,k]|,\end{split} \end{equation*} and $(L,R;k)$ is a generalized MSTD fringe pair.\\ \
	
	Case 2: Let $d \neq 0$.  Observe that $$s_iL + d_iR\ =\ [0,d_i] \cup [q, q+d_i-1]\cup \cdots \cup \{d_iq\}$$ and $$d_iL + s_iR\ =\ [0,s_i] \cup [q, q+s_i-1]\cup \cdots \cup \{s_iq\}.$$  Then $$|(s_iL+d_iR) \cap [0,k]| + |(s_iR+d_iL) \cap [0,k]|\ =\ \frac{(s_i+1)(s_i+2)}{2} + \frac{(d_i+1)(d_i+2)}{2}.$$ We have $s^2 + d^2 > \sigma^2 + \delta^2$; this follows from $0\le d < \delta \le \sigma <s$ (thus $2\sigma \delta > 2 sd$, and the claim follows from combining that inequality with $(s+d)^2 = (\sigma + \delta)^2$). Thus \begin{equation*}\begin{split}|(sL+dR) \cap [0,k]|&+|(sR+dL) \cap [0,k]| \\ &>\ |(\sigma L+\delta R) \cap [0,k]| + |(\sigma R+\delta L) \cap [0,k]|.\end{split} \end{equation*} and $(L,R;k)$ is a generalized MSTD fringe pair.
\end{proof}

\bigskip

The existence of such a fringe pair gives an infinite family of super $k$-generational MSTD sets and proves their positive density as $n$ approaches infinity.  For the sake of completeness, we give one possible construction of a super $k$-generational set from a super $k$-generational fringe pair.

\begin{proposition}
	Let $A=\{0\} \cup [k+1,2k+2] \cup (3k+3-\{0,1,q\})$ with $q>2$ and $k=q^2$. Then $A$ is a rich set with fringe pair $(L,R,k)$ for all $s + d = \sigma  + \delta \le q$ and $|A|=q^2+6$.
\end{proposition}
\begin{proof}
	It suffices to show $[k+1, 5k+5] \subseteq A + A$, which follows from
	\begin{align*}
	A+A &\ \supseteq \ [k+1,2k+2] \cup 2[k+1,2k+2] \cup ([k+1,2k+2]+\{3k+3\})\\
	& \ =\ [k+1,5k+5]. 
	\end{align*}
\end{proof}

\begin{remark}\label{super4gen}
	This construction gives a super 4-generational set with 22 elements:
	$$\lbrace 0, 17, 18, 19, 20, 21, 22, 23, 24, 25, 26, 27, 28, 29, 30, 31, 32, 33, 34, 47, 50, 51 \rbrace $$
\end{remark}

\bigskip


\section{Arbitrary differences}\label{sec:arbitrarydiff}

We now turn our attention to a simple construction for attaining specific differences between $|A+A|$ and $|A-A|$, and more generally $|sA-dA|$ and $|\sigma A-\delta A|$.  Though Martin and O'Bryant and Iyer et al. have already proved that for any integer $m$, a positive percentage of sets have the property that $|A+A|-|A-A|=m$ and a positive percentage of sets have the property that $|sA-dA|-|\sigma A-\delta A|=m$ respectively, we include the proof to advocate for the loose notion that fringe pairs are a clean perspective with which to think about arbitrary differences. We first give necessary definitions for our construction before proving Theorem \ref{thm:proportionsdsigmadelta} in \S\ref{pf:proportionsdsigmadelta}.


\subsection{Preliminaries}

\begin{definition}\label{def:affluent}
	Let $k$ and $n$ be positive integers with $2k<n$.  Let $A \subseteq [0,n]$.  We say $A$ is \emph{k-affluent} with generalized MSTD fringe pair $(L,R;k)$ if $[k+1,2n-k-1] \subseteq A+A$ and $[-n+k+1,n-k-1] \subseteq A-A$.
\end{definition}

Note that an affluent set has all middle sums and differences present; therefore, discrepancies in numbers of sums and differences are completely determined by the fringes.

\begin{proposition}\label{prop:pos_diffs}
	Given $m>0$, $L=\{0\}$, $R=[0,m] \cup \{q\}$, and $k=2q$ with $q>2m$,  if $A$ is $k$-affluent with fringe pair $(L,R;k)$, then $|A+A| - |A-A| = m$.
\end{proposition}

\begin{proof}
	Observe that $|2L \cap [0,k]| = 1$ and $|2R \cap [0,k]| = 3m + 3$ by Lemma \ref{0mq}.  Furthermore, $|(L + R) \cap [0,k]| = m+2$.  It follows that $|2L \cap[0,k]| + |2R \cap[0,k]| - 2|(L+R) \cap[0,k]| = m.$  Since $A$ is affluent, the difference between sums and differences is completely determined by the fringes.  Thus $|A+A| - |A-A| = m$.
\end{proof}

\begin{proposition}\label{prop:neg_diffs}
	Given $m>0$, $L=[0,m]$, $R=[0,m] \cup \{q\}$ and $k=2q$ with $q>2m$, if $A$ is $k$-affluent with fringe pair $(L,R; k)$,  then $|A+A| - |A-A| = -m$.
\end{proposition}

\begin{proof}
	Observe that $|2L \cap [0,k]| = 2m+1$ and $|2R \cap [0,k]| = 3m + 3$ by Lemma \ref{0mq}.  Furthermore, $(L + R) \cap [0,k] = [0,2m] \cup [q,q+m]$ and has cardinality $3m+2$.  It follows that $2|(L+R) \cap[0,k]| -(|2L \cap[0,k]| + |2R \cap[0,k]|) =m$.  Since $A$ is affluent, the difference between sums and differences is completely determined by the fringes.  Thus $|A+A|-|A-A| = -m$.
\end{proof}


Another advantage of this simple proof is that it is easily generalizable to the case of $s+d=\sigma+\delta$. This construction can be extended to attain $|sA - dA| - |\sigma A - \delta A| = (\sigma \delta - sd)m$ by again using Lemma \ref{0mq}. Then some further work shows that any arbitrary difference can be attained with a positive proportion as $n$ goes to infinity. The idea of the proof is quite simple: for positive differences $|sA-dA|-|\sigma A- \delta A|=x \geq 0$, we take the same fringe as in Proposition \ref{prop:pos_diffs}, $L=\{0\}, R=[0,m] \cup \{q\}$ with $m\geq x$. Then from $k=(s+d)q$, we simply decrease $k$ until $(L,R;k)$ has the property \begin{equation*}\begin{split}|(sL+dR) \cap [0,k]|& \ +\ |(sR+dL) \cap [0,k]| \\ &- \ |(\sigma L+\delta R) \cap [0,k]| \ - \ |(\sigma R+\delta L) \cap [0,k]| \ = \ x.\end{split} \end{equation*} This quantity is ultimately reflected in $|sA-dA|-|\sigma A-\delta A|$ if $A$ is $k$-affluent. The negative differences case is similar but for the left fringe, which we take from Proposition \ref{prop:neg_diffs} to be $L=[0,m]$.


\subsection{Proof of Theorem \ref{thm:proportionsdsigmadelta}}\label{pf:proportionsdsigmadelta}
We divide our proof into two theorems proving the positive and negative cases.

\begin{theorem}\label{thm:proportionsdsigmadeltapartI}
	Let $x \geq 0$ be a nonnegative integer and $s+d=\sigma+\delta$ with $d<\delta\leq \sigma<s$. Then the proportion of $A \subseteq [0,n]$ satisfying $|sA-dA|-|\sigma A- \delta A|=x$ is bounded below by a positive number as $n$ goes to infinity.
\end{theorem}
\begin{proof}
	Let $L=\{0\}$ and $R=[0,m] \cup \{q\}$ with $q>(s+d)m$. Then by Lemma \ref{0mq}, \[cR \ = \ [0,cm] \cup [q,q+(c-1)m] \cup \cdots \cup [(c-1)q,(c-1)q+m]\cup\{cq\},\] where we define $0R$ to be $\{0\}$. Notice that if $c<\gamma$, then $cR\subseteq \gamma R$. Now since $d<\delta\leq \sigma<s$, we have the relation \begin{equation}\label{eq:inclusion}dR\ \subseteq\ \delta R\ \subseteq\ \sigma R \ \subseteq\ sR.\end{equation} Now define the quantity \[f(k) \ =\ |dR \cap [0,k]| + |sR \cap [0,k]| - |\delta R \cap [0,k]| - |\sigma R \cap [0,k]|.\] Since $L=\{0\}$, the quantity $f(k)$ is also equal to \begin{equation*}\begin{split}|(sL+dR) \cap [0,k]|&+ \ |(sR+dL) \cap [0,k]| \\ &- |(\sigma L+\delta R) \cap [0,k]| \ - \ |(\sigma R+\delta L) \cap [0,k]|.\end{split} \end{equation*}
	
	We have the following implications:
	\begin{enumerate}
		\item $k \in dR \implies f(k) - f(k-1) \ = \ 0$,
		\item $k \in \delta R, \not \in dR \implies f(k)-f(k-1) \ = \ -1$,
		\item $k \in \sigma R, \not \in \delta R \implies f(k)-f(k-1) \ = \ 0$,
		\item $k \in sR, \not \in \sigma R \implies f(k)-f(k-1) \ = \ 1$,
		\item $k \not \in sR \implies f(k)-f(k-1) \ = \ 0$.
	\end{enumerate}
	
	Therefore, $f(k)$ and $f(k-1)$ differ at most by one. Since $|cR|=(mc+2)(c+1)/2$,
	\begin{align*}
	f(sq)& \ = \ \frac{(md+2)(d+1)}{2}+\frac{(ms+2)(s+1)}{2}-\frac{(m\delta+2)(\delta+1)}{2}-\frac{(m\sigma+2)(\sigma+1)}{2} \\
	& \ = \ \frac{1}{2}m(d(d+1)+s(s+1)-\delta(\delta+1)-\sigma(\sigma+1)) \\
	& \ = \ \frac{1}{2}(d^2+s^2-\delta^2-\sigma^2)m \\
	& \ = \ (\sigma \delta - sd)m.
	\end{align*}
	Since $0\le d < \delta \le \sigma <s$ implies $\sigma \delta - sd \geq 1$, $f(sq)\geq m$.
	
	Now we want an integer $k$ for which $f(k)=0$. For any $c>0$, $cR \cap [0,q-1]=[0,cm]$. Therefore, $f(q-1)=dm+1 + sm+1 -\delta m -1 -\sigma m -1=0$.
	
	We have the facts that $|f(k)-f(k-1)| \leq 1$, $f(q-1)=0$, and $f(sq)=(\sigma \delta - sd)m\geq m$. Therefore, for any given difference $x$ such that $0 \leq x \leq (\sigma \delta - sd)m$, there exists a $k$ such that $q-1 \leq k \leq sq$ and $f(k)=x$.
	
	Pick an arbitrary $x \geq 0$ and let $m\geq x$. Then let $q-1$ be large enough so that for a sufficiently large $n$, the proportion of $q$-affluent sets in $[0,n]$ is positive (see Lemma \ref{lemma:missing_diff}). Consider the fringe pair $(L,R;k)$ where $q-1 \leq k \leq sq$ and $f(k)=x$. Then the proportion of affluent sets with fringe pair $(L,R;k)$ is positive and bounded below as $n$ goes to infinity.
\end{proof}

Similarly, we can prove that for any $x<0$, a positive proportion of sets have $|sA-dA|-|\sigma A- \delta A| \ = \ x$. However, we must use a different fringe pair.

\begin{lemma}\label{lemma:neg_diffs}
	Let $0 \leq d<\delta \leq \sigma <s$ and $s+d=\sigma+\delta$ and $L=[0,m]$, $R=[0,m] \cup \{q\}$ where $q>(s+d) m$. Then \begin{multline*}sL+dR \ = \ [0,(s+d)m] \cup [q, q+(s+d-1)m] \cup \cdots \cup \\ [(d-1)q,(d-1)q+(s+1)m] \cup [dq,dq+sm]\end{multline*}  with \[|sL+dR|+|dL+sR|-|\sigma L +\delta R| - |\delta L + \sigma R|\ =\ (sd-\sigma \delta)m \leq -m.\] Furthermore, \[sL+dR \ \subseteq\ \sigma L + \delta R\ \subseteq\ \delta L + \sigma R\ \subseteq \ dL+sR.\]
\end{lemma}
\begin{proof}
	The first claim is a routine calculation. The set $sL+dR$ can be written as
	\begin{equation*}[0,sm]+([0,dm] \cup [q,q+(d-1)m] \cup \cdots \cup [(d-1)q,(d-1)q+m] \cup \{dq\}),\end{equation*}
	which can be simplified to the desired expression. Next, we calculate $|sL+dR|+|dL+sR|-|\sigma L +\delta R| - |\delta L + \sigma R|$.
	\begin{align*}
	|sL+dR|&\ = \ ((s+d)m+1) + ((s+d-1)m+1) + \cdots +(sm+1) \\
	&\ = \ \frac{1}{2}((s+d)m+2)(s+d+1) - \frac{1}{2}((s-1)m+2)s.
	\end{align*}
	Therefore $|sL+dR|+|dL+sR|=((s+d)m+2)(s+d+1) - \frac{1}{2}((s-1)m+2)s - \frac{1}{2}((d-1)m+2)d$. Similarly, $|\sigma L+\delta R|+|\delta L+\sigma R|=((\sigma +\delta )m+2)(\sigma+\delta+1) - \frac{1}{2}((\sigma-1)m+2)\sigma - \frac{1}{2}((\delta-1)m+2)\delta$. Since $s+d=\sigma+\delta$,
	\begin{align*}
	g((s+d)q)&\ = \ \frac{1}{2}\left(((\sigma-1)m+2)\sigma + ((\delta-1)m+2)\delta\right. \\ & \hspace{4mm}\left. \ \ - ((s-1)m+2)s - ((d-1)m+2)d\right) \\
	&\ = \ \frac{1}{2}(\sigma^2+\delta^2-s^2-d^2)m\\
	&\ = \ (sd-\sigma \delta)m \leq -m.
	\end{align*}
	
	For the last claim, notice that $\sigma L+\delta R=(sL+dR) \cup [(d+1)q,(d+1)q+(s-1)m] \cup \cdots \cup [\delta q, \delta q+\sigma m]$. So $\sigma < s$ implies that $\sigma L + \delta R \supseteq sL+dR$. The chain of inclusions follows.
\end{proof}

\begin{theorem}\label{thm:proportionsdsigmadeltapartII}
	Let $x < 0$ be a negative integer and $s+d=\sigma+\delta$ with $0\le d < \delta \le \sigma <s$. Then the proportion of $A \subseteq [0,n]$ satisfying $|sA-dA|-|\sigma A- \delta A|=x$ is bounded below by a positive number as $n$ goes to infinity.
\end{theorem}

\begin{proof}
	Let $L=[0,m]$ and $R=[0,m] \cup \{q\}$ with $q>(s+d)m$. 
	Define a similar quantity \[g(k)\ = \ |(sL+dR) \cap [0,k]| + |(dL+sR) \cap [0,k]| - |(\sigma L+\delta R) \cap [0,k]| - |(\delta L + \sigma R) \cap [0,k]|.\] It suffices to show three things: $|g(k)-g(k-1)| \leq 1$, $g(q-1)=0$, and $g((s+d)q)=(sd-\sigma \delta)m \leq -m$. We begin by claiming that $|g(k)-g(k-1)| \leq 1$.
	
	By Lemma \ref{lemma:neg_diffs}, the chain of inclusions is \[sL+dR\ \subseteq\ \sigma L + \delta R\ \subseteq\ \delta L + \sigma R\ \subseteq\ dL+sR,\] and we have the implications
	\begin{enumerate}
		\item $k \in sL+dR \implies g(k)-g(k-1) \ = \ 0$,
		\item $k \in \sigma L + \delta R, \not \in sL+dR \implies g(k)-g(k-1) \ = \ -1$,
		\item $k \in \delta L + \sigma R, \not \in \sigma L + \delta R \implies g(k)-g(k-1) \ = \ 0$,
		\item $k \in dL+sR, \not \in \delta L + \sigma R \implies g(k)-g(k-1) \ = \ 1$,
		\item $k \not \in dL + sR \implies g(k)-g(k-1) \ = \ 0$.
	\end{enumerate}
	
	\bigskip
	
	Now consider $g((s+d)q)$. By Lemma \ref{lemma:neg_diffs}, $g((s+d)q) \ = \ (sd-\sigma \delta)m \leq -m$.
	
	\bigskip
	
	Finally, consider $g(q-1)$. The intersection of each set $\left(sL+dR\right)$, $\left(\sigma L + \delta R\right)$, $\left(\delta L + \sigma R\right)$, $\left(dL+sR\right)$ with $[0,\ q-1]$ is $[0,\ (s+d)m]$. Therefore $g((s+d)m) \ = \ 0$. The rest of the proof is identical to the positive arbitrary differences proof.
\end{proof}

Together, Theorems \ref{thm:proportionsdsigmadeltapartI} and \ref{thm:proportionsdsigmadeltapartII} imply Theorem \ref{thm:proportionsdsigmadelta}.

\section{Another MSTD construction and $\frac{\log|sA-dA|}{\log|\sigma A-\delta A|}$}\label{sec:ratio}

\subsection{New Construction and Large Ratio}

We have used the fringe $L=\{0\}, R=[0,m] \cup \{q\}$ for various $m$ and $q$ in the previous sections. The choice $L=\{0\}$ has a nice property: if we fix the left fringe to be the singleton $\{0\}$, almost any right fringe gives an MSTD fringe pair. If we let $R$ be a uniform random subset and $k$ a suitably large integer (say $2 \cdot \max R$), then it is likely that $|(R+R) \cap [0,k]|+1 > 2|R \cap [0,k]|$, in which case we have an MSTD fringe pair. We make this statement precise below.

\begin{proposition}
	Let $R$ be a uniform random subset of $[0,r]$ with $r \geq 1$. Then $$\Pr(|R+R|+1>2|R|)\ > \ 1- \ \frac{(r+1)^3-(r+1)^2}{2^{r+2}}.$$
\end{proposition}

\begin{proof}
	Recall that $|R+R|=2|R|-1$ if and only if $R$ is an arithmetic progression, and $|R+R|>2|R|-1$ otherwise. The proof is straightforward. Write the elements of $R$ as $x_1 < x_2 < \cdots < x_r$ (with $r=|R|$); as the claim is trivial for arithmetic progressions, we may assume $R$ is not an arithmetic progression. We are left with proving that we cannot have $|R+R| = 2|R| - 1$ for such $R$. We proceed by contradiction. Note $x_1 + x_1 < \cdots < x_1 + x_r$ $<$ $x_2 + x_r$ $<$ $x_3 + x_r$ $<$ $\cdots$ $<$ $x_{r-1} + x_r$ $<$ $x_r + x_r$. We have just listed $2|R| - 1$ distinct elements, and thus all other sums must be in this list. In particular, $x_2+x_{r-1}$ is less than $x_2 + x_r$ but more than $x_1 + x_{r-1}$; it must therefore equal $x_1 + x_r$, which implies $x_2-x_1 = x_r - x_{r-1}$. We then note $x_2 + x_{r-1} < x_3 + x_{r-1} < x_3 + x_r$, and thus $x_3 + x_{r-1} = x_2 + x_r$, which implies $x_3-x_2 = x_r - x_{r-1}$. Arguing similarly shows all adjacent differences are equal, proving the set is an arithmetic progression.
	
	Thus it suffices to bound the number of arithmetic progressions in $[0,r]$. There are ${r+1 \choose 2}$ pairs $(i,j)$ such that $i<j$. For each $(i,j)$, there are at most $\tau(j-i)$ arithmetic progressions starting with $i$ and ending with $j$, where $\tau(x)$ is the number of divisors of $x$. Since $\tau(j-i)\leq r+1$, we have that there are at most $(r+1){r+1 \choose 2}$ arithmetic progressions in $[0,r]$. Thus $\Pr(|R+R|+1>2|R|) \ge 1-\frac{(r+1)^3-(r+1)^2}{2^{r+2}}$.
\end{proof}

Sending $r$ to infinity yields an almost sure method to construct MSTD sets. One consequence of this construction is a set $A$ for the current largest value of $\log |A+A|/\log |A-A|$. We obtained the following $R$ by a random search through subsets of $[0,90]$, fixing $0$ and picking each element with probability $0.27.$ We then sifted through the set to add or discard obvious elements, finding \begin{multline}\label{R} R \ = \ \{0, 1, 2, 4, 5, 9, 10, 12, 23, 26, 32, 38, \\ 47, 53, 59, 61, 65, 76, 78, 79, 81, 85, 86, 88, 89\}.\end{multline}

Let $A=\{0\} \cup [k+1,n-k-1] \cup (n-R)$ for $k=2 \cdot 89,\ n=3k+2$. Then we get $\log |A+A|/\log |A-A|=1.02313$. The previous largest value of $\log |A+A|/\log |A-A|$ was $1.0208$, achieved by setting $$A\ =\ \{0, 1, 2, 4, 5, 9, 12, 13, 17, 20, 21, 22, 24, 25, 29, 32, 33, 37, 40, 41, 42, 44, 45\}$$ as found in \cite{MO}. In fact, sets $A$ for which $\log |A+A|/\log |A-A|>1.0208$ are relatively common: a random search through sets $R\subseteq [0,90]$, picking each element with probability 0.27 (this value was chosen as it yielded good results from our simulations), yielded 174 such sets out of 100,000. Since $R$ in \eqref{R} was found basically through a random search, a more sophisticated method may yield a larger value for $\log |A+A|/\log |A-A|$.


It seems unlikely that $\log |A+A|/\log |A-A|$ can exceed $1.1$, or even $1.05$. Is there a theoretical upper bound? A result due to Ruzsa \cite{Ru1}, also shown in \cite{GH}, states that for any finite set $A\subseteq \Z$, $$3/4\ \leq\ \frac{\log |A+A|}{\log |A-A|} \ \leq\ 4/3.$$ The upper bound of $4/3$ is still quite far above what has been seen, and it is probably the case that the upper bound is not tight.

\begin{remark}
	Our method yields values
	\begin{equation}
	\frac{\log(4k+4+|R+R|)}{\log(4k+3+2|R|)}
	\end{equation}
	where $k=2\cdot \max R$. Since $R+R \subseteq [0,k]$, we have
	\begin{equation}\label{2000}
	\frac{\log(4k+4+|R+R|)}{\log(4k+3+2|R|)}\ \leq\ \frac{\log(5k+4)}{\log(4k+3)}
	\end{equation}
	which converges to 1 as $k \to \infty$. Therefore we do not expect to find large values of $\log |A+A|/\log |A-A|$ when we pick $R$ from $[0,r]$ with $r$ large. In fact, by \eqref{2000} we know that the maximum value of $\log |A+A|/\log |A-A|$ found with our method has $r<2000$.
\end{remark}

The above analysis suggests that the value $\log|A+A|/\log|A-A|$ tends to 1 as the size of $A$ grows. We formalize that statement below.


\subsection{Proof of Theorem \ref{thm:strongerratio}} \label{pf:strongerratio}
Propositions \ref{prop:logpart1-1} and \ref{prop:logpartt1-2} together prove the full inequality (equation \eqref{eq:log1}) in Theorem \ref{thm:strongerratio}. Proposition \ref{prop:logpart2} proves the second half of the theorem, equation \eqref{eq:log2}.

\begin{proposition}\label{prop:logpart1-1}
	Let $0\leq d < \delta \leq \sigma < s$ with $s+d=\sigma + \delta$. For every $\epsilon>0$ and a uniform random subset $A\subseteq [0,n]$, \be \lim_{n\to\infty} \Pr \left(\frac{\log|sA-dA|}{\log|\sigma A-\delta A|} \ >\ 1+\epsilon \right) \ = \ 0. \ee
\end{proposition}

The idea is that for a set $A$ with generalized MSTD fringe pair $(L,R;k)$, $A+A$ and $A-A$ are almost always full in the middle as $n$ goes to infinity and then $k$ goes to infinity. We must first introduce a lemma from Zhao \cite[Lemma 2.13]{Zh2}.

\begin{lemma}[Zhao]\label{lemma:missing_diff}
	Let $n, \overline{k}$ be positive integers with $n>2\overline{k}$. Let $A$ be a uniform random subset of $[0,n]$. Then $$\Pr ([\overline{k}+1,2n-\overline{k}-1] \not \subseteq A+A)\ \leq\ \frac{3(3/4)^{\overline{k}/2}}{2-\sqrt{3}}$$ and $$\Pr ([-n+\overline{k}+1,n-\overline{k}-1]\not \subseteq A-A])\ \leq\ 8\left(\frac{3}{4}\right)^{\overline{k}+2} + (n+1)\left(\frac{3}{4}\right)^{(n-1)/3}.$$
\end{lemma}

We now set up the proof of Proposition \ref{prop:logpart1-1}. Recall that a set $A$ is $k$-affluent if $[k+1,2n-k-1]\subseteq A+A$ and $[-n+k+1,n-k-1]\subseteq A-A$. Fix $\eps>0$ and let \begin{enumerate}
	\item \begin{multline*}\mu_n(L,R;k) \ =\ 2^{-n+1}|\{A \subseteq [0,n] : 0,n \in A, \\ A \text{ is affluent with fringe } (L,R;k), \tfrac{\log|sA-dA|}{\log|\sigma A-\delta A|}\ > \ 1+\eps\}|,\end{multline*}
	
	\item \[\lambda_n\ = \ 2^{-n+1}|\{A \subseteq [0,n] : 0,n \in A, \tfrac{\log|sA-dA|}{\log|\sigma A-\delta A|}\ >\ 1+\eps\}|,\]
	
	\item \[\Lambda_n \ = \ 2^{-n-1}|\{A \subseteq [0,n] : \tfrac{\log|sA-dA|}{\log|\sigma A-\delta A|} > 1+ \eps\}|.\]
\end{enumerate}

\ \\
Our final goal is showing that $\lim_{n\to \infty}\Lambda_n=0$. To show this, we prove that \[0 \ = \ \sum_{(L,R;k)}\mu_n(L,R;k) \ = \ \lim_{n \to \infty} \lambda_n\ = \ \lim_{n \to \infty} \Lambda_n,\] where $\sum_{(L,R;k)}\mu_n(L,R;k)$ is summing over all minimal MSTD fringe pairs $(L,R;k)$.

\begin{lemma}\label{lemma:aff}
	If $\mu=\lim_{n \to \infty} \sum_{(L,R;k)}\mu_n(L,R;k)$ exists, then $\lim_{n \to \infty} \lambda_n=\mu$.
\end{lemma}

\begin{proof}
	First notice that for any $\kbar<n$, \[\sum_{\kbar}\mu_n(L,R;k) \ := \ \sum_{\substack{(L,R;k) \\ k \leq \kbar}}\mu_n(L,R;k)  \ \leq\ \lambda_n.\] We may say this because of Lemma \ref{lemma:minimal}. We want to bound above the number of sets which are counted in $\lambda_n$ but not in $\sum_{\kbar}\mu_n(L,R;k)$. Suppose $A \subseteq [0,n]$, $0,n \in A$, and $\log |A+A|/\log |A-A|>1+\eps$ but $A$ is not $k$-affluent for any $k \leq \kbar$. Let $L=A \cap [0,\kbar], R=(n-A)\cap [0,\kbar]$. There are two cases.
	
	In the first case, $(L,R;\kbar)$ is not a generalized MSTD fringe pair. Since $A$ is generalized MSTD, $A-A$ must be missing at least one element in $[-n+\kbar+1,n-\kbar-1]$. In the second case, $(L,R;\kbar)$ is a generalized MSTD fringe pair. Then either $A+A \not \supseteq [\kbar+1, 2n-\kbar-1]$ or $A-A \not \supseteq [-n+\kbar +1, n-\kbar-1]$. So the probability that $A$ is counted in $\lambda_n$ but not in $\sum_{\kbar}\mu_n(L,R;k)$ is bounded above by
	\begin{multline*}
	\Pr(A+A \not \supseteq [\kbar+1, 2n-\kbar-1]) + \Pr(A-A \not \supseteq [-n+\kbar +1, n-\kbar-1]) \\ \leq\ \frac{3(3/4)^{\overline{k}/2}}{2-\sqrt{3}} + 8\left(\frac{3}{4}\right)^{\overline{k}+2} + (n+1)\left(\frac{3}{4}\right)^{(n-1)/3}
	\end{multline*}
	by Lemma \ref{lemma:missing_diff}. Let $n$ go to infinity to get \begin{multline*}\sum_{\kbar}\mu(L,R;k)\ \leq\ \lim \inf_{n \to \infty} \lambda_n\ \leq\ \\ \lim \sup_{n \to \infty} \lambda_n\ \leq\ \sum_{\kbar}\mu(L,R;k) + \frac{3(3/4)^{\overline{k}/2}}{2-\sqrt{3}} + 8\left(\frac{3}{4}\right)^{\overline{k}+2}.\end{multline*} Let $\kbar$ go to infinity to get $\mu=\sum_{(L,R;k)}\mu(L,R;k) = \lim_{n \to \infty} \lambda_n$.
\end{proof}

\begin{proposition}
	For every $\eps>0$, $\mu=0$.
\end{proposition}

\begin{proof}
	Fix $n$ and suppose $A\subseteq [0,n]$, $0,n \in A$, $A$ is affluent for some $k<n/2$ and $\log |A+A|/\log |A-A| >1+\eps$. Let $L=A \cap[0,k], R=(n-A) \cap [0,k]$. Then $sA-dA$ contains $[-dn+k+1,sn-k-1]$ and $\sigma A-\delta A$ contains $[-\delta n+k+1,\sigma n-k-1]$. This imposes strict bounds on $\log |sA-dA|/\log |\sigma A-\delta A|$. Note that the quantities $|(sR+dL)\cap [0,k]|$, $|(dL+sR)\cap [0,k]|$ are both bounded above by $k+1$. Observe the inequalities $|sA-dA|\leq (s+d)n-2k-1 + 2k+2$ and $|\sigma A-\delta A| \geq (\sigma+\delta)n-2k-1 + (s+d+1)$ (since $-dn,\ldots,sn \in \sigma A+\delta A$). Thus
	\begin{align}
	\frac{\log|sA-dA|}{\log|\sigma A-\delta A|} & \ \leq\ \frac{\log ((s+d)n-2k-1 + 2k+2)}{\log((\sigma+\delta)n-2k-1 + (s+d+1))} \nonumber \\
	&\ \leq\ \frac{\log((s+d)n+1)}{\log((s+d-1)n+s+d)}, \label{eq:infact}
	\end{align}
	which goes to 1 as $n$ goes to infinity. There is an $N$ such that $\log((s+d)n+1)/\log((s+d-1)n+s+d) < 1+\eps$ for all $n>N$. Therefore $\mu_n(L,R;k)$ is zero for all $n>N$. This implies that $\mu=\sum_{(L,R;k)}\mu(L,R;k)$ counts finitely many sets, so $\mu=0$.
\end{proof}

To get rid of the stipulation that $0$ and $n$ are in $A$, we must show that $\lim_{n \to \infty}\Lambda_n = \lim_{n \to \infty} \lambda_n$. This result may be found in \cite[Lemma 2.15]{Zh2}. This concludes the proof of Proposition \ref{prop:logpart1-1}. The key point is showing that the proportion of affluent sets satisfying our property is the same as the proportion of all sets satisfying our property, an idea first developed in \cite{Zh2}. The pairs $(s,d)$ and $(\sigma,\delta)$ may be reversed to show that

\begin{proposition} \label{prop:logpartt1-2}
	Let $0\leq d<\delta \leq \sigma < s$ with $s+d=\sigma+\delta$. For every $\epsilon>0$ and a uniform random subset $A\subseteq [0,n]$, \be \lim_{n\to\infty} \Pr \left(\frac{\log|sA-dA|}{\log|\sigma A-\delta A|} \ < \ 1-\epsilon \right) \ = \ 0.\ee
\end{proposition}

In fact, \eqref{eq:infact} shows that there are no affluent sets $A$ for which $|sA-dA|/|\sigma A-\delta A|\geq (s+d)/(s+d-1)$. Similarly, there are no affluent sets $A$ for which $|sA-dA|/|\sigma A-\delta A|\leq (s+d-1)/(s+d)$. Then we may say the following.

\begin{proposition} \label{prop:logpart2}
	Let $0\leq d<\delta \leq \sigma < s$ with $s+d=\sigma+\delta$. For a uniform random subset $A\subseteq [0,n]$, \be \lim_{n\to\infty} \Pr \left(\frac{s+d-1}{s+d}\ <\ \frac{|sA-dA|}{|\sigma A-\delta A|}\ <\ \frac{s+d}{s+d-1} \right) \ = \ 1. \ee
\end{proposition}

\begin{remark}
	If $s+d=\sigma+\delta \geq 3$, then all of the instances of $k$-affluent sets in the proof above may be replaced with $k$-rich sets.
\end{remark}

\section{Bi-MSTD Sets}\label{sec:bimstd}

We use similar techniques as in the previous sections to resolve the following question: \emph{Is there a decomposition of $[0,n]$ into two disjoint MSTD sets?} We show that for sufficiently large $n$, such decompositions do exist and, in fact, are a positive proportion of all decompositions of $[0,n]$ as $n$ goes to infinity. Rather than starting with an interval $[0,n]$ and decomposing it, we consider the cases in which both $A$ and $A^c:=[\min A, \max A] \setminus A$ are MSTD. We first begin with some terminology.

\subsection{Definitions}

\begin{definition}
	We say that a set $A\subseteq \Z$ is \emph{bi-MSTD} if $A$ and $A^c$ are both MSTD.
\end{definition}

An example of a bi-MSTD set in $[0,19]$ is \begin{equation}\label{smallest_bi} A\ = \ \{0, 1, 3, 7, 8, 10, 11, 12, 15, 17, 18, 19\}.\end{equation} Both $A$ and its complement $A^c=\{2, 4, 5, 6, 9, 13, 14, 16\}$ are MSTD. Notice that $A^c-2$ is the smallest MSTD set up to translation and dilation, as proved by Hegarty in \cite{He}. By an exhaustive computer search, we determined that there are no bi-MSTD sets in $[0,18]$. To show that a positive proportion of subsets of $[0,n]$ are bi-MSTD, we use the framework developed by Zhao in \cite{Zh2}. We must first develop an analogue of fringe pairs for bi-MSTD sets.

\begin{definition}
	A fringe pair $(L,R;k)$ is \emph{bi-MSTD} if \begin{equation}\begin{split}|(L+L) \cap [0,k]|+|(R+R) \cap [0,k]| & \ >\ 2|(L+R) \cap [0,k]|, \\ |(L^c+L^c) \cap [0,k]|+|(R^c+R^c) \cap [0,k]| & \ >\ 2|(L^c+R^c) \cap [0,k]|,\end{split} \end{equation} where $L^c=[0,k]\setminus L$ and $R^c=[0,k]\setminus R$.
\end{definition}

Recall from Definition \ref{defn:richset} that a set $A$ which has MSTD fringe pair $(L,R;k)$ and satisfies the property \[A+A\ \supseteq\ [k+1,2n-k-1]\] is called \emph{rich}, and a rich set with an MSTD fringe pair is MSTD. Therefore, if $A$ has a bi-MSTD fringe pair and satisfies \begin{equation}\label{eq:bi} \begin{split} A+A\ \supseteq\ [k+1,2n-k-1], \\ A^c+A^c\ \supseteq\ [k+1,2n-k-1],\end{split} \end{equation} then $A$ is bi-MSTD. We call such sets \emph{bi-rich}. We now show that the set of $A\subseteq [0,n]$ which are bi-rich has density one as $n$ goes to infinity.


\subsection{Proof of Theorem \ref{thm:bimstd}}\label{subsec:bimstd}

\begin{lemma}
	Let $0<2k<n$ and $A \subseteq[0,n]$ be a uniform random subset (dropping the condition that $0,n \in A$). Then \[\Pr( [k+1,2n-k-1] \not \subseteq A+A \text{ or } [k+1,2n-k-1] \not \subseteq A^c+A^c)\ \leq \ \frac{6(3/4)^{k/2}}{2-\sqrt{3}}.\]
\end{lemma}
\begin{proof}
	We apply the union bound. The expression above is simply two times the expression given in Lemma \ref{lemma:missing_diff}.
\end{proof}
Note that for $k = 30$, \[\Pr( [k+1,2n-k-1] \not \subseteq A+A \text{ or } [k+1,2n-k-1] \not \subseteq A^c+A^c) \ \leq\ 0.299.\] Thus a positive proportion of $A \subseteq [0,n]$ are $k$-bi-rich (in fact, most of them are). Now we simply need to find a bi-MSTD pair $(L,R;k)$ with $k = 30$. We give an example below:
\begin{equation}\label{k=30}
\begin{split}
L\ =\ \{0, 1, 2, 5, 8, 10, 11, 12, 14, 15, 16, 18, 23, 25, 26, 28, 29\}, \\
R\ =\ \{0, 1, 3, 4, 8, 10, 11, 13, 14, 15, 17, 19, 20, 22, 23, 24, 28\}.
\end{split}
\end{equation}

With this fringe pair, we have shown that for some sufficiently large $N > 60$, for any $n>N$, the proportion of subsets of $[0,n]$ which are bi-MSTD is bounded below by a positive number. Since the bi-MSTD set in \eqref{smallest_bi} is contained in $[0,19]$, the same holds for $n \geq 19$.


\section{Future work}\label{sec:future}







We end with a list of some interesting additional questions to pursue.

\begin{enumerate}
	\item We were able to construct a fringe pair for which $|sA-dA|>|\sigma A-\delta A|$ for $s+d=\sigma +d+2 \leq q$ and $s>\sigma $. Can we construct a fringe pair to satisfy $|sA-dA|>|\sigma A-\delta A|$ for $s+d=\sigma +d+2 \leq q$ for any sequence $\{(s,d,\sigma ,\delta )_i : s+d=\sigma +\delta =i, 2 \leq i \leq q\}$ without using base expansion?
	
	\item As $n$ goes to infinity, $|sA-dA|/|\sigma A-\delta A|$ is between $(s+d-1)/(s+d)$ and $(s+d)/(s+d-1)$ almost all the time. For any $\eps>0$, is there a positive proportion of $A$ such that $t < |sA-dA|/|\sigma A-\delta A| < t+ \eps$ for $(s+d-1)/(s+d) < t < (s+d)/(s+d-1)-\eps$? What if instead of uniform random subsets $A\subseteq [0,n]$, we have the probability of $j \in A \subseteq[0,n]$ decaying with $n$?
	
	\item We showed that a positive proportion of subsets of $[0,n]$ are bi-MSTD, which is equivalent to showing that a positive proportion of decompositions of $[0,n]$ into two sets have the property that both sets are MSTD. Can we decompose $[0,n]$ into three sets which are MSTD? For any finite number $k$, is there a sufficiently large $n$ for which there is a $k$-decomposition into MSTD sets?
	
	\item We found bi-MSTD fringe pairs by random searches and were not able to come up with any deterministic algorithm to produce bi-MSTD fringe pairs. Are there clean families of bi-MSTD fringe pairs? What about generalized bi-MSTD fringe pairs?
	
	\item We have shown that any arithmetic progression can be decomposed into two MSTD sets, and such decompositions have positive density among all decompositions of that arithmetic progression as $n$ goes to infinity. What are the sets which can be decomposed into two MSTD sets? With positive density?
	
	\item In our proof that a positive proportion of subsets in $[0,n]$ are bi-MSTD, we showed that the proportion of bi-rich sets goes to one as $n$ goes to infinity. How quickly does the density go to one? How quickly does the proportion of MSTD sets that are bi-MSTD increase with $n$?
\end{enumerate}

\section*{Acknowledgments}

Our research was conducted as part of the SMALL 2015 REU program. The first, second, and fourth named authors were supported by Williams College and NSF grant DMS-1347804. The third named author was supported by NSF grant DMS-1265673. The first named author was also supported by a Clare Boothe Luce scholarship. We thank Gal Gross and Kevin O'Bryant for many enlightening conversations and the referee for constructive comments.

\end{document}